\documentclass[12pt,centertags,oneside,reqno]{amsart}
\usepackage{amsmath,amsthm,amssymb,bm}

\usepackage{amscd,amsxtra,calc}
\usepackage{cmmib57}
\usepackage{url}

\usepackage{amscd}
\usepackage{pstricks}
\usepackage{color}
\usepackage[a4paper,width=16.2cm,top=3cm,bottom=3cm]{geometry}

\numberwithin{equation}{section}

\theoremstyle{plain}

\newtheorem{thm}{Theorem}[section]
\newtheorem{theorem}[thm]{Theorem}
\newtheorem{thma}{Theorem}

\newtheorem*{corb}{Corollary B}
\newtheorem{corc}{Corollary}

\newtheorem{lemma}[thm]{Lemma}
\newtheorem{corollary}[thm]{Corollary}
\newtheorem{proposition}[thm]{Proposition}

\theoremstyle{definition}

\newtheorem{remark}[thm]{Remark}

\newtheorem{definition}[thm]{Definition}

\newtheorem{example}[thm]{Example}

\newtheorem{conjecture}[thm]{Conjecture}

\numberwithin{equation}{section}

\newtheorem{report}{report}
\newtheorem{exercise}{excercise}[section]

\renewcommand{\labelenumi}{\rm{(}\arabic{enumi}\rm{)}}

\usepackage{mleftright}
\usepackage{mdwlist}

\usepackage{etoolbox}

\makeatletter
\def\subsection{\@startsection{subsection}{1}%
  \z@{.5\linespacing\@plus.7\linespacing}{-.5em}%
  {\normalfont\itshape}}
\patchcmd{\@sect}
  {\@addpunct.}
  {}
  {}
  {}
\makeatother

\allowdisplaybreaks

\makeatletter
\def\subsection{\@startsection{subsection}{2}%
  \z@{.5\linespacing\@plus.7\linespacing}{.3\linespacing}%
  {\normalfont\bfseries}}
\makeatother

\usepackage{listings}
\usepackage{caption}

\usepackage{amsmath}

\usepackage{mathtools}
\DeclarePairedDelimiter{\abs}{\lvert}{\rvert}

\newcommand{\id}{\mathrm{id}}

\usepackage[pdftex]{graphicx}

\usepackage{threeparttable}

\usepackage{appendix}

\title[]{A {\large$1$}-cohomologically hyperbolic birational map of {\large$\mathbb{P}^3$}, with a transcendental arithmetic degree}

\author{Yutaro Sugimoto
}
\begin{document}
\maketitle

\begin{abstract}
 We construct a $1$-cohomologically hyperbolic birational map of $\mathbb{P}^3$, with transcendental first dynamical degree.
 The arithmetic degree of this map at a $\overline{\mathbb{Q}}$-point is transcendental.
\end{abstract}

\section{Introduction}
 Let $X$ be a smooth projective variety over an algebraic closed field $k$ with characteristic $0$.
 For a dominant rational map $f:X\dashrightarrow X$, define the $p$-th dynamical degree of $f$ as
 \begin{align*}
  \lambda_p(f)=\lim_{n\to\infty}((f^n)^{*}(H^p)\cdot H^{\mathrm{dim}X-p})^{\frac{1}{n}}
 \end{align*}
 for $0\leq p\leq \mathrm{dim} X$, where $H$ is an ample divisor (cf.\ \cite{DS05}, \cite{Tru15}).
 In this paper, we mainly consider the case $k=\overline{\mathbb{Q}}$.
 Birational maps, whose first dynamical degree is transcendental, are studied in \cite{BDJK24}, and they give an explicit example in that paper (cf.\ Section \ref{Birational maps}).\par
 Next, we introduce the arithmetic degree.
 Define the set
 \begin{align*}
  X(\overline{\mathbb{Q}})_f=\left\{x\in X(\overline{\mathbb{Q}})\;\middle|\; f^n(x) \text{ can be defined for all } n\geq0\right\}
 \end{align*}
 and for $x\in X(\overline{\mathbb{Q}})_f$, define the forward orbit of $x$ by
 \begin{align*}
  \mathcal{O}_f(x)=\left\{f^n(x)\;\middle|\; n\geq0\right\}.
 \end{align*}
 Let $h_X:X(\overline{\mathbb{Q}})\rightarrow[0,\infty)$ be a fixed logarithmic Weil height function (cf.\ \cite[Part B]{HS00}) and we set $h_X^{+}=\mathrm{max}\{1,h_X\}$. 
 Define the upper and lower arithmetic degree of $f$ at $x\in X(\overline{\mathbb{Q}})_f$ by
 \begin{align*}
  \overline{\alpha}_f(x)=\limsup_{n\to\infty} h_X^{+}(f^n(x))^{\frac{1}{n}}\\
  \underline{\alpha}_f(x)=\liminf_{n\to\infty} h_X^{+}(f^n(x))^{\frac{1}{n}}
 \end{align*}
 and if these values are equal, the value is called the arithmetic degree, and we write $\alpha_f(x)$.
 The next conjecture is proposed in \cite[Conjecture 1]{KS14}.
 \begin{conjecture}
  Let $X$ be a smooth projective variety over $\overline{\mathbb{Q}}$, and let $f:X\dashrightarrow X$ be a dominant rational map defined over $\overline{\mathbb{Q}}$ and let $P\in X(\overline{\mathbb{Q}})_f$.
  Then,
  \begin{enumerate}
  \renewcommand{\labelenumi}{\rm{(\alph{enumi})}}
  \renewcommand{\theenumi}{\alph{enumi}}
   \item $\alpha_f(P)$ exists.
   \item If $\mathcal{O}_f(P)$ is Zariski dense in $X$, then $\lambda_1(f)=\alpha_f(P)$.
   \item\label{(c)} $\alpha_f(P)$ is an algebraic integer.
   \item\label{(d)} The collection of arithmetic degrees $\left\{\alpha_f(Q)\;\middle|\; Q\in X(\overline{\mathbb{Q}})_f\right\}$ is finite.
  \end{enumerate}
 \end{conjecture}
 A counterexample for (\ref{(d)}) is found in \cite{LS20}, when $f$ is a birational map and $X=\mathbb{P}^4$,
 Also, a counterexample for (\ref{(c)}) is found in \cite{MW22}, when $f$ is a dominant rational map and $X=\mathbb{P}^2$, but the birational case is unknown.\par
 In Section \ref{2-cohomologically hyperbolic birational map}, we calculate the second dynamical degree of the example of birational map of $\mathbb{P}^3$ (\cite[Section 7.2]{BDJK24}), whose first dynamical degree is transcendental, and prove this map is $2$-cohomologically hyperbolic.\par
 Also, in Section \ref{1-cohomologically hyperbolic birational map}, we prove the next theorem.
 \begin{thma}\label{Main theorem}
  There exists a birational map $f:\mathbb{P}_{\overline{\mathbb{Q}}}^3\dashrightarrow \mathbb{P}_{\overline{\mathbb{Q}}}^3$, which is $1$-cohomologically hyperbolic, and $\lambda_1(f)$ is transcendental.
 \end{thma}
 This map provides us with the next corollary (cf.\ Remark \ref{maximal dynamical degree is transcendental}), which is the answer for the first half of \cite[Question 1.6]{BDJK24}.
 But the last half of \cite[Question 1.6]{BDJK24} is hard to prove in this way (see Remark \ref{simultaneously}).
 \begin{corb}
  For any $d\geq3$, there exists a birational map $f:\mathbb{P}_{\overline{\mathbb{Q}}}^d\dashrightarrow \mathbb{P}_{\overline{\mathbb{Q}}}^d$, whose first dynamical degree is transcendental, and is maximal (not necessarily strictly large) among the dynamical degrees.
 \end{corb}
 By combining with the results of \cite{MW22}, we prove the next corollary in the last part of Section \ref{1-cohomologically hyperbolic birational map}, and this is a counterexample for (\ref{(c)}) on birational maps of $\mathbb{P}^3$.
 \setcounter{corc}{2}
 \begin{corc}\label{corollary of Main theorem}
  For any $d\geq3$, there exist a birational map $\phi:\mathbb{P}^d\dashrightarrow\mathbb{P}^d$ over $\overline{\mathbb{Q}}$ and a point $P\in\mathbb{P}^{d}(\overline{\mathbb{Q}})_\phi$ such that its arithmetic degree $\alpha_\phi(P)$ is transcendental.\par
  Moreover, if $d=3$, we can take $P\in\mathbb{P}^3(\overline{\mathbb{Q}})_{\phi}$ as $\overline{\mathcal{O}_\phi(P)}=\mathbb{P}^3$.
 \end{corc}
 Besides, we use SageMath \cite{sagemath} as a programming language in this paper.
 This language able us to calculate the unit group of a number field, and the absolute logarithmic height of an algebraic number.\vspace{10pt}\\
 \noindent
 {\bf Acknowledgements.} The author thanks Professor Keiji Oguiso and Long Wang for helpful advice on this paper.
 He also thanks Yuta Takada for verifying the calculation results by programming.

\section{Preliminaries}
 \subsection{Dynamical degrees}\label{Dynamical degrees}
  Let $X_1,X_2$ be smooth projective varieties, let $f:X_1\dashrightarrow X_2$ be a birational map, let $\alpha\in A_k(X_1)$ and $\beta\in A_l(X_2)$ be cycle classes with $k+l=n=\mathrm{dim}(X_1)=\mathrm{dim}(X_2)$.
  Also, we write the structure morphisms by $p_i:X_i\rightarrow S$.
  Take the Zariski closure $\Gamma_f\subset X_1\times X_2$ of the graph of $f$, and define the projections $\pi_1:X_1\times X_2\rightarrow X_1$, $\pi_2:X_1\times X_2\rightarrow X_2$.
  The pullback $f^{*}\beta$ is  defined as ${\pi_1}_{*}([\Gamma_f]\cdot\pi_2^{*}\beta)$, and the pushforward $f_{*}\alpha$ is  defined as ${\pi_2}_{*}([\Gamma_f]\cdot\pi_1^{*}\alpha)$ (cf.\ \cite[Definition 16.1.2]{Ful98}).
  Then, the intersection of cycles can be calculated as
  \begin{align*}
   {p_1}_{*}(f^{*}(\beta)\cdot\alpha)&={p_1}_{*}({\pi_1}_{*}([\Gamma_f]\cdot{\pi_2}^{*}\beta)\cdot\alpha)\\
                                     &=({p_1}_{*}\circ{\pi_1}_{*})(([\Gamma_f]\cdot{\pi_2}^{*}\beta)\cdot{\pi_1}^{*}\alpha)\\
                                     &=({p_2}_{*}\circ{\pi_2}_{*})(([\Gamma_f]\cdot{\pi_1}^{*}\alpha)\cdot{\pi_2}^{*}\beta)\\
                                     &={p_2}_{*}({\pi_2}_{*}([\Gamma_f]\cdot{\pi_1}^{*}\alpha)\cdot\beta)\\
                                     &={p_2}_{*}(f_{*}(\alpha)\cdot\beta)
  \end{align*}\medskip
  in $A_0(S)\simeq\mathbb{Z}$, by the push-pull formula.\par
  Thus, for the case $X=X_1=X_2$ and $\mathrm{dim}(X)=3$, $\lambda_0(f)=\lambda_3(f)=1$ and
  \begin{align*}
   \lambda_2(f)&=\lim_{n\to\infty}((f^n)^{*}(H\cdot H)\cdot H)^{\frac{1}{n}}\\
               &=\lim_{n\to\infty}((H\cdot H\cdot (f^n)_{*}(H))^{\frac{1}{n}}\\
               &=\lim_{n\to\infty}((H\cdot H\cdot (f^{-n})^{*}(H))^{\frac{1}{n}}\\
               &=\lambda_1(f^{-1}).
  \end{align*}\par
  For the case $X=\mathbb{P}^N$ and its rational map $f:X\dashrightarrow X$ which is represented as 
  \begin{align*}
  \begin{array}{cccc}
   f: & \mathbb{P}^N & \dashrightarrow & \mathbb{P}^N\\
   & [x_0:x_1:\cdots:x_N] & \mapsto & [\phi_0:\phi_1:\cdots:\phi_N]
  \end{array}
  \end{align*}
  for homogeneous polynomials $\phi_i$ of the same degree $d\in\mathbb{Z}_{>0}$ ($\phi_0,\phi_1,\ldots,\phi_N$ do not have any common factor).
  Then,
  \begin{align*}
   ((f)^{*}(H)\cdot H^{N-1})=(dH\cdot H^{N-1})=d
  \end{align*}
  and define $\mathrm{deg}(f):=d=\mathrm{deg}(\phi_i)$ in this paper.
  Thus, the first dynamical degree $\lambda_1(f)$ can be considered as the growth rate of the sequence $\mathrm{deg}(f),\mathrm{deg}(f^2),\ldots,\mathrm{deg}(f^n),\ldots$ and so $\lambda_1(f)\leq\mathrm{deg}(f)$.
  Now the rational map $f:\mathbb{P}_{\overline{\mathbb{Q}}}^N\dashrightarrow\mathbb{P}_{\overline{\mathbb{Q}}}^N$ can be extended to $f_\mathbb{C}:\mathbb{P}_\mathbb{C}^N\dashrightarrow\mathbb{P}_\mathbb{C}^N$ and $\lambda_1(f)=\lambda_1(f_\mathbb{C})$ holds.\par
  Also, for a birational map $f:X\dashrightarrow X$, $\lambda_0(f)=\lambda_{\mathrm{dim}(X)}(f)=1$ holds (cf.\ \cite{Tru15}).\par
  It is known that dynamical degrees are birational invariant, which can be stated as below.
  \begin{theorem}[{cf.\ \cite[Corollaire 7]{DS05}, \cite[Corollary 1.2]{DN11}, \cite[Lemma 3.7]{Tru15}}]\label{birational conjugate}
   Let $X$ and $Y$ be smooth projective varieties of the same dimension $n$ over an algebraic closed field of characteristic zero.
   Let $f:X\dashrightarrow X$ and $g:Y\dashrightarrow Y$ be dominant rational maps and let $\pi:X\dashrightarrow Y$ be a birational map with $\pi\circ f=g\circ\pi$.
   Then, $\lambda_p(f)=\lambda_p(g)$ for all $0\leq p\leq\mathrm{dim}(X)$.
  \end{theorem}

 \subsection{Cohomologically hyperbolicity}
  In \cite{MW22}, the next notation is introduced (cf.\ \cite{Gue10}).
  \begin{definition}[{\cite[Definition 1.1]{MW22}}]
   Let $X$ be a projective variety over an algebraic closed field of characteristic zero.
   Let $f:X\dashrightarrow X$ be a dominant rational self-map.
   For $1\leq p\leq \mathrm{dim}(X)$, we say $f$ is $p$-cohomologically hyperbolic if
   \begin{align*}
    \lambda_p(f)>\lambda_i(f)
   \end{align*}
   for all $i\in\{0,1,\ldots,\mathrm{dim}(X)\}\backslash\{p\}$.\par
  \end{definition}
  In \cite{MW22}, the next theorem is stated in relation to cohomologically hyperbolicity.
  \begin{theorem}[{\cite[Theorem 1.11]{MW22}}]\label{MW22}
   Let $X$ be a smooth projective variety defined over $\overline{\mathbb{Q}}$.
   Let $f:X\dashrightarrow X$ be a $1$-cohomologically hyperbolic dominant rational map.
   Then, there is $x\in X(\overline{\mathbb{Q}})_f$ such that the forward orbit $\mathcal{O}_f(x)$ is Zariski dense in $X$.
   Moreover, we can take such $x$ as that $\alpha_f(x)$ exists and $\alpha_f(x)=\lambda_1(f)$.
  \end{theorem}

 \subsection{Birational maps of $\mathbb{P}^N$}\label{birational maps on P^N}\label{Birational maps}
  In this subsection, we introduce some birational maps of $\mathbb{P}^N$ and the results of \cite{BDJK24}.\par
  First, we consider the case for which the ground field is $k$ with $\mathrm{char}(k)=0$.
  \begin{itemize}
   \item For $M\in\mathrm{GL}_{N+1}(k)$, define the linear transformation $L_M:\mathbb{P}_k^N\rightarrow\mathbb{P}_k^N$ as its coordinate transformation by $M$, and this is an automorphism.
   \item For $A=(a_{ij})_{i.j}\in\mathrm{GL}_{N}(\mathbb{Z})$, define the monomial map $h_A:\mathbb{P}_k^N\rightarrow\mathbb{P}_k^N$ as
         \begin{align*}
          \left(\frac{x_1}{x_0},\ldots,\frac{x_N}{x_0}\right)\mapsto\left(\left(\frac{x_1}{x_0}\right)^{a_{11}}\ldots\left(\frac{x_N}{x_0}\right)^{a_{1N}},\ldots,\left(\frac{x_1}{x_0}\right)^{a_{N1}}\ldots\left(\frac{x_N}{x_0}\right)^{a_{NN}}\right)
         \end{align*}
         by corresponding
         \begin{align*}
          \left(\frac{x_1}{x_0},\ldots,\frac{x_N}{x_0}\right)\quad\text{with}\quad[x_0:x_1:\cdots:x_N].
         \end{align*}
  \end{itemize}
  Now $(L_M)^{-1}=L_{M^{-1}}$ and $(h_A)^{-1}=h_{A^{-1}}$ hold.\par
  The next result about the calculation of the dynamical degree is proved in \cite{BDJK24}.
  In \cite{BDJK24}, this proposition is proved also for $\mathbb{P}^d$ ($d\geq3$), that is, the general case.
  \begin{proposition}[{cf.\ \cite[Section 3]{BDJK24}}]\label{equation for dynamical degree}
   Define the matrix
   \begin{align*}
    B:=\begin{pmatrix}
        1 & -1 & 1 & -1 \\ 1 & 1 & -1 & 1 \\ -1 & 1 & 1 & -1 \\ 1 & -1 & 1 & 1
       \end{pmatrix}.
   \end{align*}
   For a matrix $A\in\mathrm{GL}_{3}(\mathbb{Z})$, define
   \begin{align*}
    f_A:=L_{B^{-1}}\circ h_{-I}\circ L_{B}\circ h_A:\mathbb{P}^3\dashrightarrow\mathbb{P}^3.
   \end{align*}
   Define the finite sets of $3$-dimensional vectors as
   \begin{align*}
    \mathcal{U}:=\left\{\begin{pmatrix}
                         0 \\ 0 \\ 0 
                        \end{pmatrix},
			\begin{pmatrix}
                         -1 \\ 0 \\ 0 
                        \end{pmatrix},
			\begin{pmatrix}
                         0 \\ -1 \\ 0 
                        \end{pmatrix},
			\begin{pmatrix}
                         0 \\ 0 \\ -1 
                        \end{pmatrix}\right\}\\
    \mathcal{V}:=\left\{\begin{pmatrix}
                         1 \\ 1 \\ 0 
                        \end{pmatrix},
			\begin{pmatrix}
                         0 \\ 1 \\ 1 
                        \end{pmatrix},
			\begin{pmatrix}
                         -1 \\ -1 \\ 0 
                        \end{pmatrix},
			\begin{pmatrix}
                         0 \\ -1 \\ -1 
                        \end{pmatrix}\right\}\\
    \mathcal{P}:=\left\{\begin{pmatrix}
                         -1 \\ -1 \\ -1 
                        \end{pmatrix},
			\begin{pmatrix}
                         1 \\ 0 \\ 0 
                        \end{pmatrix},
			\begin{pmatrix}
                         0 \\ 1 \\ 0 
                        \end{pmatrix},
			\begin{pmatrix}
                         0 \\ 0 \\ 1 
                        \end{pmatrix}\right\}.
   \end{align*}
   Then, if any element of $A^n(\mathcal{V}\cup\mathcal{P})$ ($n\geq1$) is contained in the interior of a $3$-dimensional cone generated by elements of $\mathcal{P}$, the first dynamical degree $\lambda=\lambda_1(f_A)$ satisfies
   \begin{align*}
    \sum_{n=1}^{\infty} \left(\sum_{v\in\mathcal{V}}\max_{u\in\mathcal{U}}\langle u,A^nv\rangle\right)\cdot\frac{1}{\lambda^n}=1
   \end{align*}
   where $\langle\cdot,\cdot\rangle$ is the inner product.
  \end{proposition}
  \begin{remark}
   The function
   \begin{align*}
   \begin{array}{cccc}
    \phi: & \mathbb{R}^3 & \rightarrow & \mathbb{R}\\
    & v & \mapsto & \max_{u\in\mathcal{U}}\langle u,v\rangle
   \end{array}
   \end{align*}
   is non-negative and is linear in each $3$-dimensional cone generated by elements of $\mathcal{P}$.
   By writing
   \begin{align*}
    \Psi_{\mathcal{U},\mathcal{V}}(A)=\sum_{v\in\mathcal{V}}\max_{u\in\mathcal{U}}\langle u,Av\rangle,
   \end{align*}
   there is a unique $\lambda>0$ which satisfies the equation
   \begin{align*}
    \sum_{n=1}^{\infty} \Psi_{\mathcal{U},\mathcal{V}}(A^n)\cdot\frac{1}{\lambda^n}=1
   \end{align*}
   and this is the dynamical degree.
  \end{remark}
  In \cite{BDJK24}, birational maps with transcendental dynamical degree are constructed.
  Some of the notations are needed for writing down the condition for transcendency.\par
  For the matrix $A$ in Proposition \ref{equation for dynamical degree}, let $P(x)$ be the characteristic polynomial of $A$, and let $\xi$ be one of its roots that has a maximal absolute value.
  Assume $P(x)$ is irreducible and denote the splitting field of $P(x)$ over $\mathbb{Q}$ by $K$.
  Let $H_{\xi}$ be the subspace of $\mathbb{C}^3$, which is generated by the eigenvector associated with $\xi$.
  Define the projection
  \begin{align*}
    \pi:\mathbb{C}^3\rightarrow\mathrm{H}_{\xi}
   \end{align*}
  and the set
  \begin{align*}
    \mathcal{W}:=(\mathcal{U}-\mathcal{U})\backslash\{0\}=\left\{\pm\begin{pmatrix}
                                                           1 \\ 0 \\ 0 
                                                          \end{pmatrix},
		                                  	  \pm\begin{pmatrix}
                                                           0 \\ 1 \\ 0 
                                                          \end{pmatrix},
		                                       	  \pm\begin{pmatrix}
                                                           0 \\ 0 \\ 1 
                                                          \end{pmatrix},
		                                  	  \pm\begin{pmatrix}
                                                           1 \\ -1 \\ 0 
                                                          \end{pmatrix},
		                                  	  \pm\begin{pmatrix}
                                                           0 \\ 1 \\ -1 
                                                          \end{pmatrix},
		                                       	  \pm\begin{pmatrix}
                                                           -1 \\ 0 \\ 1 
                                                          \end{pmatrix}\right\}.\\
  \end{align*}
  For each $v\in\mathcal{V},w\in\mathcal{W}$, define 
  \begin{align*}
   \sigma(v,w):=-\frac{\overline{\langle w,\pi(v)\rangle}}{\langle w,\pi(v)\rangle}\in K.
  \end{align*}
  The sufficient condition for the transcendency of $\lambda>0$ which satisfies the equation 
  \begin{align*}
    \sum_{n=1}^{\infty} \Psi_{\mathcal{U},\mathcal{V}}(A^n)\cdot\frac{1}{\lambda^n}=1
   \end{align*}
  is stated as below (cf.\ \cite[Theorem 1.3, Remark 5.11]{BDJK24}).
  \begin{enumerate}
  \renewcommand{\labelenumi}{\rm{(\roman{enumi})}}
  \renewcommand{\theenumi}{\roman{enumi}}
   \item $P(x)$ is irreducible
   \item the roots of $P(x)$ are $\xi_1=\xi,\xi_2=\overline{\xi},\xi_3$ with $\abs{\xi}>1$ and the angle of $\xi_1$ is not a rational multiple of $\pi$
   \item $\sigma(v,w)\notin U_K$
   \item $\frac{\sigma(v,w)}{\sigma(v',w')}\notin U_K$ for $v,v'\in\mathcal{V}$, $w,w'\in\mathcal{W}$ unless both pairs of vectors are linearly dependent (over $\mathbb{R}$)
  \end{enumerate}
  Now $U_K$ is the set of the invertible elements of $\mathcal{O}_K$. 
  In \cite{BDJK24}, the matrix $A$, which satisfies the above condition is constructed, and the associated birational map has a transcendental dynamical degree.
  \begin{theorem}[{cf.\ \cite[Section 7.2]{BDJK24}}]\label{transcendental dynamical degree}
   Define the matrix
    \begin{align*}
     A:=\begin{pmatrix}
         -3 & -14 & -12 \\ 4 & 11 & 6 \\ -2 & -4 & -1
        \end{pmatrix}.
    \end{align*}
    and $B$ as in Proposition \ref{equation for dynamical degree}.
    Then, in terms of
    \begin{align*}
     f_A:=L_{B^{-1}}\circ h_{-I}\circ L_{B}\circ h_A:\mathbb{P}^3\dashrightarrow\mathbb{P}^3,
    \end{align*}
    the dynamical degree $\lambda_1(f_A)$ is transcendental.
  \end{theorem}

\section{A $2$-cohomologically hyperbolic birational map}\label{2-cohomologically hyperbolic birational map}
 In this section, we calculate the dynamical degrees of the map $f_A$ of the form in Proposition \ref{equation for dynamical degree}.\par
 For the birational map $f_A$ in Theorem \ref{transcendental dynamical degree}, $\lambda_0(f_A)=\lambda_3(f_A)=1$ and  $\lambda_1(f_A)$ is transcendental.\par
 We denote
  \begin{align*}
   A=\begin{pmatrix}
      -3 & -14 & -12 \\ 4 & 11 & 6 \\ -2 & -4 & -1
     \end{pmatrix}\text{ and }
   B=\begin{pmatrix}
      1 & -1 & 1 & -1 \\ 1 & 1 & -1 & 1 \\ -1 & 1 & 1 & -1 \\ 1 & -1 & 1 & 1
     \end{pmatrix}.
  \end{align*}
  First, we consider the size of $\lambda_1(f)$.
  $\lambda=\lambda_1(f_A)$ satisfies the equation
  \begin{align*}
   \sum_{n=1}^{\infty} \Psi_{\mathcal{U},\mathcal{V}}(A^n)\cdot\frac{1}{\lambda^n}=1
  \end{align*}
  and each $\Psi_{\mathcal{U},\mathcal{V}}(A^n)$ is non-negative.
  Thus, $\lambda\geq\Psi_{\mathcal{U},\mathcal{V}}(A)$ and by
  \footnotesize
  \begin{align*}
   \Psi_{\mathcal{U},\mathcal{V}}(A)&=\sum_{v\in\mathcal{V}}\max_{u\in\mathcal{U}}\langle u,Av\rangle\\
                                    &=\max_{u\in\mathcal{U}}\left\langle u,\begin{pmatrix}
                                                                       -17 \\ 15 \\ -6 
                                                                      \end{pmatrix}\right\rangle
                                     +\max_{u\in\mathcal{U}}\left\langle u,\begin{pmatrix}
                                                                       -26 \\ 17 \\ -5 
                                                                      \end{pmatrix}\right\rangle
                                     +\max_{u\in\mathcal{U}}\left\langle u,\begin{pmatrix}
                                                                       17 \\ -15 \\ 6 
                                                                      \end{pmatrix}\right\rangle
                                     +\max_{u\in\mathcal{U}}\left\langle u,\begin{pmatrix}
                                                                       26 \\ -17 \\ 5 
                                                                      \end{pmatrix}\right\rangle\\
                                    &=17+26+15+17=75,
  \end{align*}
  \normalsize
  it derives $\lambda_1(f_A)=\lambda\geq75$.
  Also, by $f_A=L_{B^{-1}}\circ h_{-I}\circ L_{B}\circ h_A$,
  \begin{align*}
   \lambda_1(f_A)\leq\mathrm{deg}(f_A)\leq\mathrm{deg}(L_{B^{-1}})\cdot\mathrm{deg}(h_{-I})\cdot\mathrm{deg}(L_{B})\cdot\mathrm{deg}(h_A).
  \end{align*}
  Now $h_A:\mathbb{P}^3\dashrightarrow\mathbb{P}^3$ can be represented as
  \begin{align*}
   [x_0:x_1:x_2:x_3]\mapsto[x_0^{21}x_1^{3}x_2^{14}x_3^{12}:x_0^{50}:x_1^{7}x_2^{25}x_3^{18}:x_0^{28}x_1x_2^{10}x_3^{11}]
  \end{align*}
  and so $\mathrm{deg}(h_A)=50$.
  Also, $h_{-I}$ is represented as
  \begin{align*}
   [x_0:x_1:x_2:x_3]\mapsto[x_1x_2x_3:x_0x_2x_3:x_0x_1x_3:x_0x_1x_2]
  \end{align*}
  and so $\mathrm{deg}(h_{-I})=3$.
  Thus,
  \begin{align*}
   \lambda_1(f_A)\leq1\cdot3\cdot1\cdot50=150
  \end{align*}
  and in summary, $75\leq\lambda_1(f_A)\leq150$.\par
  Next, we calculate the second dynamical degree $\lambda_2(f_A)$.
  As in Section \ref{Dynamical degrees}, $\lambda_2(f_A)=\lambda_1({f_A}^{-1})$ and ${f_A}^{-1}$ can be represented as 
  \begin{align*}
   {f_A}^{-1}=h_{A^{-1}}\circ L_{B^{-1}}\circ h_{-I}\circ L_{B} 
  \end{align*}
  and this is birational conjugate to
  \begin{align*}
   f_{A^{-1}}=L_{B^{-1}}\circ h_{-I}\circ L_{B}\circ h_{A^{-1}}
  \end{align*}
  by the next diagram.
  \begin{figure}[h]
   \includegraphics[scale=0.4]{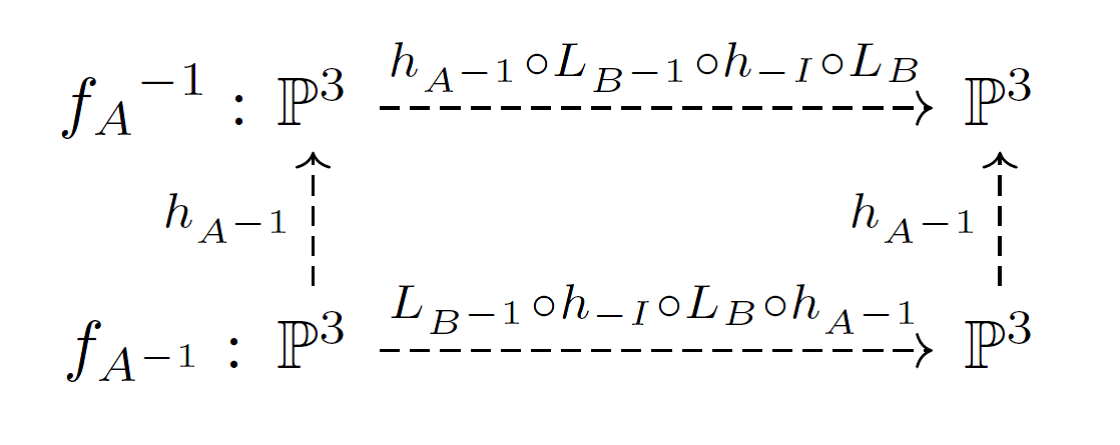}
   \label{figure}
  \end{figure}\newpage
  Thus, by Theorem \ref{birational conjugate}, $\lambda_2(f_A)=\lambda_1({f_{A^{-1}}})$ and the next lemma holds.
  \begin{lemma}\label{condition for the recurrence}
   For all $n\in\mathbb{Z}_{>0}$, any element of $A^{-n}(\mathcal{V}\cup\mathcal{P})$ is contained in the interior of a $3$-dimensional cone generated by elements of $\mathcal{P}$.
  \end{lemma}
  \begin{proof}
   Now
   \begin{align*}
    A^{-1}=\begin{pmatrix}
            13 & 34 & 48 \\ -8 & -21 & -30 \\ 6 & 16 & 23
           \end{pmatrix}.
   \end{align*}
   By computation, $A^{-1}$ has three eigenvalues
   \begin{align*}
    \sigma_1=14.5227\cdots,\ \sigma_2=0.2386\cdots+0.1091\cdots i,\ \sigma_3=0.2386\cdots-0.1091\cdots i.
   \end{align*}
   For these eigenvalues, the elements of $\mathcal{V}$ and $\mathcal{P}$ are decomposed into the eigenspaces as below:
   \footnotesize
   \begin{align*}
    &v_1=\begin{pmatrix}
     1 \\ 1 \\ 0 
    \end{pmatrix}=\begin{pmatrix}
                   3.2278\cdots \\ -2.0244\cdots \\ 1.5364\cdots 
                  \end{pmatrix}+
		  \begin{pmatrix}
                   -1.1139\cdots-2.9971\cdots i \\ 1.5122\cdots+1.4692\cdots i \\ -0.7682\cdots-0.2464\cdots i
                  \end{pmatrix}+
		  \begin{pmatrix}
                   -1.1139\cdots+2.9971\cdots i \\ 1.5122\cdots-1.4692\cdots i \\ -0.7682\cdots+0.2464\cdots i
                  \end{pmatrix}\\
    &v_2=\begin{pmatrix}
     0 \\ 1 \\ 1 
    \end{pmatrix}=\begin{pmatrix}
                   5.7087\cdots \\ -3.5804\cdots \\ 2.7172\cdots 
                  \end{pmatrix}+
		  \begin{pmatrix}
                   -2.8543\cdots-2.0846\cdots i \\ 2.2902\cdots+0.4329\cdots i \\ -0.8586\cdots+0.2410\cdots i
                  \end{pmatrix}+
		  \begin{pmatrix}
                   -2.8543\cdots+2.0846\cdots i \\ 2.2902\cdots-0.4329\cdots i \\ -0.8586\cdots-0.2410\cdots i
                  \end{pmatrix}\\
    &-v_1=\begin{pmatrix}
     -1 \\ -1 \\ 0 
    \end{pmatrix}=\begin{pmatrix}
                   -3.2278\cdots \\ 2.0244\cdots \\ -1.5364\cdots 
                  \end{pmatrix}+
		  \begin{pmatrix}
                   1.1139\cdots+2.9971\cdots i \\ -1.5122\cdots-1.4692\cdots i \\ 0.7682\cdots+0.2464\cdots i
                  \end{pmatrix}+
		  \begin{pmatrix}
                   1.1139\cdots-2.9971\cdots i \\ -1.5122\cdots+1.4692\cdots i \\ 0.7682\cdots-0.2464\cdots i
                  \end{pmatrix}\\
    &-v_2=\begin{pmatrix}
     0 \\ -1 \\ -1 
    \end{pmatrix}=\begin{pmatrix}
                   -5.7087\cdots \\ 3.5804\cdots \\ -2.7172\cdots 
                  \end{pmatrix}+
		  \begin{pmatrix}
                   2.8543\cdots+2.0846\cdots i \\ -2.2902\cdots-0.4329\cdots i \\ 0.8586\cdots-0.2410\cdots i
                  \end{pmatrix}+
		  \begin{pmatrix}
                   2.8543\cdots-2.0846\cdots i \\ -2.2902\cdots+0.4329\cdots i \\ 0.8586\cdots+0.2410\cdots i
                  \end{pmatrix}\\
    &v'_1=\begin{pmatrix}
     -1 \\ -1 \\ -1 
    \end{pmatrix}=\begin{pmatrix}
                   -6.5853\cdots \\ 4.1302\cdots \\ -3.1345\cdots 
                  \end{pmatrix}+
		  \begin{pmatrix}
                   2.7926\cdots+3.1849\cdots i \\ -2.5651\cdots-1.1056\cdots i \\ 1.0672\cdots-0.0572\cdots i
                  \end{pmatrix}+
		  \begin{pmatrix}
                   2.7926\cdots-3.1849\cdots i \\ -2.5651\cdots+1.1056\cdots i \\ 1.0672\cdots+0.0572\cdots i
                  \end{pmatrix}\\
    &v'_2=\begin{pmatrix}
     1 \\ 0 \\ 0 
    \end{pmatrix}=\begin{pmatrix}
                   0.8765\cdots \\ -0.5497\cdots \\ 0.4172\cdots 
                  \end{pmatrix}+
		  \begin{pmatrix}
                   0.0617\cdots-1.1002\cdots i \\ 0.2748\cdots+0.6726\cdots i \\ -0.2086\cdots-0.1838\cdots i
                  \end{pmatrix}+
		  \begin{pmatrix}
                   0.0617\cdots+1.1002\cdots i \\ 0.2748\cdots-0.6726\cdots i \\ -0.2086\cdots+0.1838\cdots i
                  \end{pmatrix}\\
    &v'_3=\begin{pmatrix}
     0 \\ 1 \\ 0 
    \end{pmatrix}=\begin{pmatrix}
                   2.3512\cdots \\ -1.4746\cdots \\ 1.1191\cdots 
                  \end{pmatrix}+
		  \begin{pmatrix}
                   -1.1756\cdots-1.8968\cdots i \\ 1.2373\cdots+0.7965\cdots i \\ -0.5595\cdots-0.0626\cdots i
                  \end{pmatrix}+
		  \begin{pmatrix}
                   -1.1756\cdots+1.8968\cdots i \\ 1.2373\cdots-0.7965\cdots i \\ -0.5595\cdots+0.0626\cdots i
                  \end{pmatrix}\\
    &v'_4=\begin{pmatrix}
     0 \\ 0 \\ 1 
    \end{pmatrix}=\begin{pmatrix}
                   3.3575\cdots \\ -2.1057\cdots \\ 1.5981\cdots 
                  \end{pmatrix}+
		  \begin{pmatrix}
                   -1.6787\cdots-0.1878\cdots i \\ 1.0528\cdots-0.3636\cdots i \\ -0.2990\cdots+0.3036\cdots i
                  \end{pmatrix}+
		  \begin{pmatrix}
                   -1.6787\cdots+0.1878\cdots i \\ 1.0528\cdots+0.3636\cdots i \\ -0.2990\cdots-0.3036\cdots i
                  \end{pmatrix}
   \end{align*}
   \normalsize
   To confirm that the vector ${\footnotesize\begin{pmatrix} x \\ y \\ z \end{pmatrix}}\in\mathbb{R}^3$ is located in the interior of a $3$-dimensional cone generated by elements of $\mathcal{P}$, it is sufficient to check
   \begin{align*}
    (\ast)\  x\neq0 , y\neq0 , z\neq0 , x\neq y , y\neq z \text{ and } z\neq x.
   \end{align*}
   Thus, the conditions for $-v_1$ and $-v_2$ are reduced to the conditions for $v_1$ and $v_2$.
   It remains to check the condition ($\ast$) for $A^{-n}(v_1),A^{-n}(v_2),A^{-n}(v'_1),A^{-n}(v'_2),A^{-n}(v'_3)$ and $A^{-n}(v'_4)$ for all $n\geq1$ and now we prove only for $A^{-n}(v_1)$ (the other cases are proved in the same way).\par
   The equation
   \footnotesize
   \begin{align*}
    A^{-n}(v_1)=\begin{pmatrix}
                 X_n \\ Y_n \\ Z_n 
                \end{pmatrix}=\sigma_1^n\begin{pmatrix}
                                          3.2278\cdots \\ -2.0244\cdots \\ 1.5364\cdots 
                                         \end{pmatrix}
			     +\sigma_2^n\begin{pmatrix}
                                          -1.1139\cdots-2.9971\cdots i \\ 1.5122\cdots+1.4692\cdots i \\ -0.7682\cdots-0.2464\cdots i
                                         \end{pmatrix}
			     +\sigma_3^n\begin{pmatrix}
                                          -1.1139\cdots+2.9971\cdots i \\ 1.5122\cdots-1.4692\cdots i \\ -0.7682\cdots+0.2464\cdots i
                                         \end{pmatrix}                      
   \end{align*}
   \normalsize
   means that the coordinates in the first-term growth to $\pm\infty$, and the coordinates in the other terms converge to $0$ by $n\to\infty$.
   Write
   \footnotesize
   \begin{gather*}
    \sigma_1^n\begin{pmatrix}
                3.2278\cdots \\ -2.0244\cdots \\ 1.5364\cdots 
               \end{pmatrix}=\begin{pmatrix}
                              x_n \\ y_n \\ z_n 
                             \end{pmatrix},\\
    \sigma_2^n\begin{pmatrix}
                -1.1139\cdots-2.9971\cdots i \\ 1.5122\cdots+1.4692\cdots i \\ -0.7682\cdots-0.2464\cdots i
               \end{pmatrix}=\begin{pmatrix}
                              x'_n \\ y'_n \\ z'_n 
                             \end{pmatrix},
    \sigma_3^n\begin{pmatrix}
                -1.1139\cdots+2.9971\cdots i \\ 1.5122\cdots-1.4692\cdots i \\ -0.7682\cdots+0.2464\cdots i
               \end{pmatrix}=\begin{pmatrix}
                              x''_n \\ y''_n \\ z''_n 
                             \end{pmatrix}
   \end{gather*}
   \normalsize
   with $X_n=x_n+x'_n+x''_n$, $Y_n=y_n+y'_n+y''_n$, $Z_n=z_n+z'_n+z''_n$ and then
   \begin{align*} 
    \min\{\abs{x_n},\abs{y_n},\abs{z_n},\abs{x_n-y_n},\abs{y_n-z_n},\abs{z_n-x_n}\}=\abs{\sigma_1}^n\cdot1.5364\cdots,\\
    \max\{\abs{x'_n},\abs{y'_n},\abs{z'_n},\abs{x'_n-y'_n},\abs{y'_n-z'_n},\abs{z'_n-x'_n}\}=\abs{\sigma_2}^n\cdot5.1812\cdots,\\
    \max\{\abs{x''_n},\abs{y''_n},\abs{z''_n},\abs{x''_n-y''_n},\abs{y''_n-z''_n},\abs{z''_n-x''_n}\}=\abs{\sigma_3}^n\cdot5.1812\cdots
   \end{align*}
   and so
   \begin{align*} 
    \min\{\abs{X_n},\abs{Y_n},\abs{Z_n},\abs{X_n-Y_n},\abs{Y_n-Z_n},\abs{Z_n-X_n}\}\geq19.5936\cdots>0.
   \end{align*}
   Thus, the condition ($\ast$) is true for $A^{-n}(v_1)$, and the others are confirmed similarly. 
  \end{proof}
  Thus, by Proposition \ref{equation for dynamical degree}, $\lambda'=\lambda_1(f_{A^{-1}})$ satisfies the equation
  \begin{align*}
   \sum_{n=1}^{\infty} \Psi_{\mathcal{U},\mathcal{V}}(A^{-n})\cdot\frac{1}{{\lambda'}^n}=\sum_{n=1}^{\infty} \left(\sum_{v\in\mathcal{V}}\max_{u\in\mathcal{U}}\langle u,A^{-n}v\rangle\right)\cdot\frac{1}{{\lambda'}^n}=1.
  \end{align*}
  The first few terms of the sequences $A^{-n}v$ for $v\in\mathcal{V}$ are as below:
  \footnotesize
  \begin{align*}
   &A^{-1}v_1=\begin{pmatrix}
               47 \\ -29 \\ 22 
             \end{pmatrix},A^{-2}v_1=\begin{pmatrix}
                                      681 \\ -427 \\ 324 
                                     \end{pmatrix},A^{-3}v_1=\begin{pmatrix}
                                                              9887 \\ -6201 \\ 4706
                                                             \end{pmatrix},\ldots\\
   &A^{-1}v_2=\begin{pmatrix}
              82 \\ -51 \\ 39
             \end{pmatrix},A^{-2}v_2=\begin{pmatrix}
                                      1204 \\ -755 \\ 573
                                     \end{pmatrix},A^{-3}v_2=\begin{pmatrix}
                                                              17486 \\ -10967 \\ 8323
                                                             \end{pmatrix},\ldots\\
   &A^{-1}(-v_1)=\begin{pmatrix}
               -47 \\ 29 \\ -22 
             \end{pmatrix},A^{-2}(-v_1)=\begin{pmatrix}
                                         -681 \\ 427 \\ -324 
                                        \end{pmatrix},A^{-3}(-v_1)=\begin{pmatrix}
                                                                    -9887 \\ 6201 \\ -4706
                                                                   \end{pmatrix},\ldots\\
   &A^{-1}(-v_2)=\begin{pmatrix}
                 -82 \\ 51 \\ -39
                \end{pmatrix},A^{-2}(-v_2)=\begin{pmatrix}
                                            -1204 \\ 755 \\ -573
                                           \end{pmatrix},A^{-3}(-v_2)=\begin{pmatrix}
                                                                       -17486 \\ 10967 \\ -8323
                                                                      \end{pmatrix},\ldots
  \end{align*}
  \normalsize
  Similar to the proof of Lemma \ref{condition for the recurrence}, we can confirm that, for $n\geq1$, $\langle u,A^{-n}v_1\rangle$ are all maximized by $u={\footnotesize\begin{pmatrix}0 \\ -1 \\ 0\end{pmatrix}}$, $\langle u,A^{-n}v_2\rangle$ are all maximized by $u={\footnotesize\begin{pmatrix}0 \\ -1 \\ 0\end{pmatrix}}$, $\langle u,A^{-n}(-v_1)\rangle$ are all maximized by $u={\footnotesize\begin{pmatrix}-1 \\ 0 \\ 0\end{pmatrix}}$ and $\langle u,A^{-n}(-v_2)\rangle$ are all maximized by $u={\footnotesize\begin{pmatrix}-1 \\ 0 \\ 0\end{pmatrix}}$.
  Thus,
  \begin{align*}
   &\max_{u\in\mathcal{U}}\langle u,A^{-k}v_1\rangle=29,427,6201,90055,\ldots(k=1,2,3,4,\ldots)\\
   &\max_{u\in\mathcal{U}}\langle u,A^{-k}v_2\rangle=51,755,10967,159271,\ldots(k=1,2,3,4,\ldots)\\
   &\max_{u\in\mathcal{U}}\langle u,A^{-k}(-v_1)\rangle=47,681,9887,143585,\ldots(k=1,2,3,4,\ldots)\\
   &\max_{u\in\mathcal{U}}\langle u,A^{-k}(-v_2)\rangle=82,1204,17486,253944,\ldots(k=1,2,3,4,\ldots)
  \end{align*}
  and these sequences are linear recurrences that satisfy $a_n=15a_{n-1}-7a_{n-2}+a_{n-3}$, associated with $x^3-15x^2+7x-1$, the characteristic polynomial of $A^{-1}$.
  Thus, $P_n=\Psi_{\mathcal{U},\mathcal{V}}(A^{-n})$ is also a linear recurrence that satisfies $P_n=15P_{n-1}-7P_{n-2}+P_{n-3}$, and the first few terms are 
  \begin{align*}
   209,3067,44541,646855\ldots.
  \end{align*}
  Therefore, the equation is modified as
  \begin{align*}
   1=\sum_{n=1}^{\infty} \frac{P_n}{{\lambda'}^n}
  \end{align*}
  and
  \begin{align*}
   &\lambda'-209=\sum_{n=1}^{\infty} \frac{P_{n+1}}{{\lambda'}^n}\\
   &{\lambda'}^2-209\lambda'-3067=\sum_{n=1}^{\infty} \frac{P_{n+2}}{{\lambda'}^n}\\
   &{\lambda'}^3-209{\lambda'}^2-3067\lambda'-44541=\sum_{n=1}^{\infty} \frac{P_{n+3}}{{\lambda'}^n}.
  \end{align*}
  By using the equation $P_{n+3}=15P_{n+2}-7P_{n+1}+P_n$,
  \begin{align*}
   {\lambda'}^3-209{\lambda'}^2-3067\lambda'-44541=15({\lambda'}^2-209\lambda'-3067)-7(\lambda'-209)+1
  \end{align*}
  and so, ${\lambda'}^3-224{\lambda'}^2+75\lambda'=0$, and $\lambda'=223.6646\cdots$.
  In summary, the birational map $f_A$ in Theorem \ref{transcendental dynamical degree} satisfies
  \begin{align*}
   \lambda_0(f_A)=1,\  75\leq\lambda_1(f_A)\leq150,\  \lambda_2(f_A)=223.6646\cdots,\ \lambda_3(f_A)=1
  \end{align*}
  with $\lambda_1(f_A)$ transcendental, and the map $f_A$ is $2$-cohomologically hyperbolic.

\section{A $1$-cohomologically hyperbolic birational map}\label{1-cohomologically hyperbolic birational map}
 In this section, we consider another birational map with transcendental (first) dynamical degree and confirm that it is $1$-cohomologically hyperbolic.\par
 Instead of the matrix $A$ in Theorem \ref{transcendental dynamical degree}, we use
 \begin{align*}
  A_1:=\begin{pmatrix}
     56 & -19 & -17 \\ -16 & 6 & 5 \\ 207 & -71 & -63
    \end{pmatrix}.
 \end{align*}
 Define the birational map $f_{A_1}=L_{B^{-1}}\circ h_{-I}\circ L_{B}\circ h_{A_1}:\mathbb{P}^3\dashrightarrow\mathbb{P}^3$ for this matrix $A_1$.
 By birationality, $\lambda_0(f_{A_1})=\lambda_3(f_{A_1})=1$.\par
 In Section \ref{Transcendency of lambda_1}, we prove that $\lambda_1(f_{A_1})$ is transcendental, and estimate its size.
 In Section \ref{Calculation of lambda_2}, we calculate $\lambda_2(f_{A_1})$ explicitly (the value is algebraic).
 \subsection{Transcendency of $\lambda_1(f_{A_1})$}\label{Transcendency of lambda_1}
  Now
  \begin{align*}
   A_1=\begin{pmatrix}
      56 & -19 & -17 \\ -16 & 6 & 5 \\ 207 & -71 & -63
     \end{pmatrix}=\begin{pmatrix}
                    1 & 1 & 3 \\ 1 & 2 & 0 \\ 2 & 1 & 10
                   \end{pmatrix}
                   \begin{pmatrix}
                    0 & 1 & 0 \\ 0 & 0 & 1 \\ 1 & 0 & -1
                   \end{pmatrix}
                   \begin{pmatrix}
                    20 & -7 & -6 \\ -10 & 4 & 3 \\ -3 & 1 & 1
                   \end{pmatrix}
  \end{align*}
  and the characteristic polynomial of $A_1$ is $x^3+x^2-1$ and its roots are
  \begin{align*}
   \xi_1=-0.8774\cdots+0.7448\cdots i,\ \xi_2=-0.8774\cdots-0.7448\cdots i, \xi_3=0.7548\cdots.
  \end{align*}
  Now $\frac{\xi_1}{\xi_2}$ is not a root of unity, since its Galois conjugate $\frac{\xi_1}{\xi_3}$ has a modulus not equal to $1$.
  Thus, the angle of $\xi_1$ is not a rational multiple of $\pi$.
  Define $K:=\mathbb{Q}(\xi_1,\xi_2)$ as the splitting field of $x^3+x^2-1$.
  Therefore, it is enough to confirm the next conditions for the transcendency of $\lambda_1(f_{A_1})$ (see Section \ref{birational maps on P^N}).
  \begin{enumerate}
  \renewcommand{\labelenumi}{\rm{(\roman{enumi})}}
  \renewcommand{\theenumi}{\roman{enumi}}
   \item\label{(i)} $\sigma(v,w)\notin U_K$
   \item\label{(ii)} $\frac{\sigma(v,w)}{\sigma(v',w')}\notin U_K$ for $v,v'\in\mathcal{V}$, $w,w'\in\mathcal{W}$ unless both pairs of vectors are linearly dependent (over $\mathbb{R}$)
   \item\label{(iii)} Each element of $A_1^n(\mathcal{V}\cup\mathcal{P})$ is contained in the interior of some $3$-dimensional cone generated by elements of $\mathcal{P}$ for all $n\in\mathbb{Z}_{>0}$.
  \end{enumerate}
  Now
  \begin{align*}
    &\mathcal{V}=\left\{\begin{pmatrix}
                         1 \\ 1 \\ 0 
                        \end{pmatrix},
			\begin{pmatrix}
                         0 \\ 1 \\ 1 
                        \end{pmatrix},
			\begin{pmatrix}
                         -1 \\ -1 \\ 0 
                        \end{pmatrix},
			\begin{pmatrix}
                         0 \\ -1 \\ -1 
                        \end{pmatrix}\right\}\\
    &\mathcal{W}=\left\{\pm\begin{pmatrix}
                        1 \\ 0 \\ 0 
                       \end{pmatrix},
		       \pm\begin{pmatrix}
                        0 \\ 1 \\ 0 
                       \end{pmatrix},
		       \pm\begin{pmatrix}
                        0 \\ 0 \\ 1 
                       \end{pmatrix},
		       \pm\begin{pmatrix}
                        1 \\ -1 \\ 0 
                       \end{pmatrix},
	               \pm\begin{pmatrix}
                        0 \\ 1 \\ -1 
                       \end{pmatrix},
		       \pm\begin{pmatrix}
                        -1 \\ 0 \\ 1 
                       \end{pmatrix}\right\}\\
    &\mathcal{P}=\left\{\begin{pmatrix}
                         -1 \\ -1 \\ -1 
                        \end{pmatrix},
			\begin{pmatrix}
                         1 \\ 0 \\ 0 
                        \end{pmatrix},
			\begin{pmatrix}
                         0 \\ 1 \\ 0 
                        \end{pmatrix},
			\begin{pmatrix}
                         0 \\ 0 \\ 1 
                        \end{pmatrix}\right\}.
  \end{align*}
  By using SageMath \cite{sagemath}, (\ref{(i)}) and (\ref{(ii)}) are confirmed as below.
\begin{lstlisting}
sage: R.<x> = PolynomialRing(QQ)
....: K.<a> = (x^3 +x^2-1).splitting_field(); K
Number Field in a with defining polynomial x^6 + 5*x^5 + 27*x^4 + 65*x^3 + 171*x^2 + 235*x + 317
\end{lstlisting}
  Here we define the splitting field $K$ of $x^3 +x^2-1$.
\begin{lstlisting}
sage: G=K.galois_group()
....: A=matrix(K,[[56,-19,-17],[-16,6,5],[207,-71,-63]])
....: sol=A.eigenmatrix_right()
....: P=sol[0]
....: z1=P[0][0]
....: z2=P[1][1]
....: z3=P[2][2]
....: for i in range (0,6):
....:     if G[i](z1)!=z1 and G[i](z2)!=z2 and G[i](z3)==z3:
....:         G0=G[i]
....: print(G0)
(1,3)(2,4)(5,6)
\end{lstlisting}
  Here we define the complex conjugate on $K$.
\begin{lstlisting}
sage: v1=matrix(K,[[1],[1],[0]])
....: v2=matrix(K,[[0],[1],[1]])
....: v3=matrix(K,[[-1],[-1],[0]])
....: v4=matrix(K,[[0],[-1],[-1]])
....: W1=matrix(K,[[1],[0],[0]])
....: W2=matrix(K,[[0],[1],[0]])
....: W3=matrix(K,[[0],[0],[1]])
....: W4=matrix(K,[[1],[-1],[0]])
....: W5=matrix(K,[[0],[1],[-1]])
....: W6=matrix(K,[[-1],[0],[1]])
....: L1=[v1,v2,v3,v4]
....: L2=[W1,W2,W3,W4,W5,W6]
\end{lstlisting}
  Here we define the sets $\mathcal{V}$ and $\mathcal{W}$.
  Half of the elements of $\mathcal{W}$ can be omitted.
\begin{lstlisting}
sage: V=matrix(K,[[1,0,0]])
....: V1=matrix(K,[[1],[0],[0]])
....: Q=sol[1]
....: S=Q^(-1)
....: L3=[]
....: L4=[]
....: UK = UnitGroup(K)
....: for i in range (0,4):
....:     v=L1[i]
....:     for j in range (0,6):
....:         W=L2[j]
....:         w=Q*V1*V*S*v
....:         s=((W.transpose())*w)[0][0]
....:         t=-G0(s)/s
....:         L3.append([t,i,j])
....:         L4.append(t not in UK)
....: print(all(L4))
True
\end{lstlisting}
  Here we confirm (\ref{(i)}).
\begin{lstlisting}
sage: L5=[]
....: for i in range (0,24):
....:     t_1=L3[i][0]
....:     for j in range (i+1,24):
....:         t_2=L3[j][0]
....:         l1=L1[L3[i][1]]
....:         l2=L1[L3[j][1]]
....:         l3=L2[L3[i][2]]
....:         l4=L2[L3[j][2]]
....:         L5.append((((l1==l2 or l1==-l2) and (l3==l4 or l3==-l4))) or (t_1/t_2 not in UK))
....: print(all(L5))
True
\end{lstlisting}
  Here we confirm (\ref{(ii)}).\par
  It remains to show (\ref{(iii)}).
  It suffices to show that the iteration of the elements of $\mathcal{V}$ and $\mathcal{P}$ by $A_1$ are not on the planes $x=0,y=0,z=0,x=y,y=z$, or $z=x$.
  Now half of the cases for the elements of $\mathcal{V}$ can be reduced to the other half.
  For each iteration, the sequences $\{x_n\}_{n\geq1},\{y_n\}_{n\geq1},\{z_n\}_{n\geq1},\{x_n-y_n\}_{n\geq1},\{y_n-z_n\}_{n\geq1}$, and $\{z_n-x_n\}_{n\geq1}$ are linear recurrences associated with $x^3+x^2-1$.
  From now, we consider these $36$ sequences.
  Each sequence can be written as $c_1\xi_1^n+c_2\xi_2^n+c_3\xi_3^n$ and we would like to prove the next lemma.
  \begin{lemma}[{cf.\ \cite[Lemma 7.4]{BDJK24}}]\label{zero terms of linear recurrence}
   Let $\{a_N\}_{N\geq1}$ be arbitrarily one of the above $36$ sequences and write
   \begin{align*}
    a_N=c_1\xi_1^N+c_2\xi_2^N+c_3\xi_3^N
   \end{align*}
   for the above $\xi_1,\xi_2,\xi_3$ and the associated $c_1,c_2,c_3\in K$.\par
   If $N>7\cdot10^{18}$, then $a_N\neq0$.
  \end{lemma}
  \begin{remark}\label{coefficients are not zero}
   By calculating specifically, we get $\{a_1,a_2,a_3,\ldots\}\neq\{0,0,0,\ldots\}$ for all sequences, and so by considering the group action of the elements of $\mathrm{Gal}(K/\mathbb{Q})$ on $a_N$, we obtain $c_i\neq0$ ($1\leq i\leq3$).
   Also, by computation, we have $c_1\neq-c_2$ for every sequence.
  \end{remark}
  To prove this lemma, we use Theorem \ref{Baker-Wüstholz theorem}.
  From now, the branch of $\mathrm{log}(z)$ is fixed by $-\pi<\mathrm{Im}(\mathrm{log}(z))\leq\pi$ .
  For a number field $K$ with $[K:\mathbb{Q}]=d$, denote the set of the places of $K$ by $M_K$.
  Now for $v\in M_K$, the valuation $\abs{\cdot}_v$ is normalized by $[K_v:\mathbb{Q}_v]$ to satisfy the product formula.
  For $\alpha\in K$, define the height function
  \begin{align*}
   h(\alpha)=\sum_{v\in M_K}\mathrm{log}(\mathrm{max}\{1,\abs{\alpha}_v\})
  \end{align*}
  and
  \begin{align*}
   h'(\alpha)=\frac{1}{d}\mathrm{max}\{h(\alpha),\abs{\mathrm{log}(\alpha)},1\}.
  \end{align*}
  Also for a linear function $L=\beta_1x_1+\cdots+\beta_nx_n$ with $\beta_1,\ldots,\beta_n\in K$, define
  \begin{align*}
   h(L):=\sum_{v\in M_K}\mathrm{log}(\mathrm{max}\{\abs{\beta_1}_v,\ldots,\abs{\beta_n}_v\}).
  \end{align*}
  and
  \begin{align*}
   h'(L)=\frac{1}{d}\mathrm{max}\{h(L),1\}.
  \end{align*}
  \begin{theorem}[{\cite{BW93}}]\label{Baker-Wüstholz theorem}
   Let $\alpha_1,\ldots,\alpha_k$ be algebraic numbers, not $0$ or $1$, and define $K:=\mathbb{Q}(\alpha_1,\ldots,\alpha_k)$ with $[K:\mathbb{Q}]=d$.
   Let $L=n_1x_1+\cdots+n_kx_k$ be a linear function with rational integer coefficients.
   If $L(\mathrm{log}(\alpha_1),\ldots,\mathrm{log}(\alpha_k))\neq0$, then,
   \begin{align*}
    \mathrm{log}\abs{L(\mathrm{log}(\alpha_1),\ldots,\mathrm{log}(\alpha_k))}&=\mathrm{log}\abs{n_1\mathrm{log}(\alpha_1)+\cdots+n_k\mathrm{log}(\alpha_k)}\\
                                                                             &>-18(k+1)!k^{k+1}(32d)^{k+2}\mathrm{log}(2kd)h'(\alpha_1)\cdots h'(\alpha_k)h'(L).
   \end{align*}
  \end{theorem}
  \begin{proof}[Proof of Lemma \ref{zero terms of linear recurrence}]
   The proof is proceeding as in \cite[Lemma 7.4]{BDJK24}.\par
   If $c_1\xi_1^N+c_2\xi_2^N+c_3\xi_3^N=0$, then
   \begin{align*}
    -c_1\left(\frac{\xi_1}{\xi_2}\right)^N-c_2=c_3\left(\frac{\xi_3}{\xi_2}\right)^N
   \end{align*}
   and so
   \begin{align*}
    \abs[\bigg]{\left(-\frac{c_1}{c_2}\right)\left(\frac{\xi_1}{\xi_2}\right)^N-1}=\abs[\bigg]{\frac{c_3}{c_2}}\abs[\bigg]{\frac{\xi_3}{\xi_2}}^N=\abs[\bigg]{\frac{c_3}{c_2}}\abs[\bigg]{\frac{1}{\xi_1}}^{3N}.
   \end{align*}
   By writing
   \begin{align*}
    \left(-\frac{c_1}{c_2}\right)\left(\frac{\xi_1}{\xi_2}\right)^N=e^{i\theta}\quad(-\pi<\theta\leq\pi)
   \end{align*}
   and using the inequality
   \begin{align*}
    \abs{e^{i\theta}-1}\geq\frac{\abs{\theta}}{2},
   \end{align*}
   this implies
   \begin{align*}
    \abs[\bigg]{\mathrm{log}\left(-\frac{c_1}{c_2}\right)+N\mathrm{log}\left(\frac{\xi_1}{\xi_2}\right)+k\cdot2\pi i}&=\abs[\bigg]{\mathrm{log}\left(\left(-\frac{c_1}{c_2}\right)\cdot\left(\frac{\xi_1}{\xi_2}\right)^N\right)}\\
                                                                                                                     &\leq2\cdot\abs{e^{i\theta}-1}\\
                                                                                                                     &=2\cdot\abs[\bigg]{\frac{c_3}{c_2}}\cdot\abs[\bigg]{\frac{1}{\xi_1}}^{3N}
   \end{align*}
   for $-N\leq k\leq N$.
   The left side is not zero by Remark \ref{coefficients are not zero}, and so by applying Theorem \ref{Baker-Wüstholz theorem} for $K=\mathbb{Q}(\xi_1,\xi_2)=\mathbb{Q}(\frac{\xi_1}{\xi_2})$,
   \begin{align*}
    &-18\cdot4!\cdot3^4\cdot(32\cdot6)^5\cdot\mathrm{log}(2\cdot3\cdot6)\cdot h'\left(\frac{\xi_1}{\xi_2}\right)h'\left(-\frac{c_1}{c_2}\right)\cdot h'(-1)\cdot\mathrm{max}\{\mathrm{log}\abs{2k},\mathrm{log}(N),1\}\\
    &<\mathrm{log}\abs[\bigg]{\mathrm{log}\left(-\frac{c_1}{c_2}\right)+N\mathrm{log}\left(\frac{\xi_1}{\xi_2}\right)+2k\mathrm{log}(-1)}
    \leq\mathrm{log}\left(2\cdot\abs[\bigg]{\frac{c_3}{c_2}}\cdot\abs[\bigg]{\frac{1}{\xi_1}}^{3N}\right)\\
    &=\mathrm{log}(2)+\mathrm{log}\abs[\bigg]{\frac{c_3}{c_2}}-3N\mathrm{log}\abs{\xi_1}.
   \end{align*}
   Now
   \begin{align*}
    h'(-1)=\frac{1}{6}\mathrm{max}\{h(-1),\abs{\mathrm{log}(-1)},1\}=\frac{\pi}{6}.
   \end{align*}
   We calculate some of the terms by using SageMath \cite{sagemath}.
\begin{lstlisting}
sage: K.<t>=NumberField(x^6 + 5*x^5 + 27*x^4 + 65*x^3 + 171*x^2 + 235*x + 317)
....: A=matrix(K,[[56,-19,-17],[-16,6,5],[207,-71,-63]])
....: sol=A.eigenmatrix_right()
....: P=sol[0]
....: z1=P[0][0]
....: z2=P[1][1]
....: z3=P[2][2]
....: print(((z1/z2).global_height())*K.absolute_degree())
0.843598722968886
\end{lstlisting}
   Thus, $h'\left(\frac{\xi_1}{\xi_2}\right)\leq\frac{1}{6}\mathrm{max}\left\{0.8435\cdots,\pi,1\right\}=\frac{\pi}{6}$.
\begin{lstlisting}
sage: Q=sol[1]
....: S=Q^(-1)
....: v1=matrix(K,[[1],[1],[0]])
....: v2=matrix(K,[[0],[1],[1]])
....: v3=matrix(K,[[-1],[-1],[0]])
....: v4=matrix(K,[[0],[-1],[-1]])
....: v5=matrix(K,[[-1],[-1],[-1]])
....: v6=matrix(K,[[1],[0],[0]])
....: w1=matrix(K,[[1,0,0]])
....: w2=matrix(K,[[0,1,0]])
....: w3=matrix(K,[[0,0,1]])
....: w4=matrix(K,[[1,-1,0]])
....: w5=matrix(K,[[0,1,-1]])
....: w6=matrix(K,[[-1,0,1]])
....: L1=[v1,v2,v3,v4,v5,v6]
....: L2=[w1,w2,w3,w4,w5,w6]
....: L3=[]
....: for i in range (0,6):
....:     v=S*L1[i]
....:     for j in range (0,6):
....:         w=L2[j]*Q
....:         c1=w[0][0]*v[0][0]
....:         c2=w[0][1]*v[1][0]
....:         c3=w[0][2]*v[2][0]
....:         h_1=((-c1/c2).global_height())*K.absolute_degree()
....:         L3.append(h_1)
....: max(L3)
38.9601692717445
\end{lstlisting}
   Thus, $h'\left(-\frac{c_1}{c_2}\right)\leq\frac{1}{6}\mathrm{max}\left\{38.9601\cdots,\pi,1\right\}<7$.
\begin{lstlisting}
sage: A=matrix([[56,-19,-17],[-16,6,5],[207,-71,-63]])
....: sol=A.eigenmatrix_right()
....: P=sol[0]
....: z1=P[0][0]
....: z2=P[1][1]
....: z3=P[2][2]
....: for i in range (0,2):
....:     z=P[i][i]
....:     if abs(z)<1:
....:         break
....: i0=i
....: for i in range (0,2):
....:     z=P[i][i]
....:     if abs(z)>1:
....:         break
....: i1=i
....: print(log(float(abs(P[i1][i1]))))
0.14059978716148083
\end{lstlisting}
   Thus, $3N\mathrm{log}\abs{\xi_1}>0.42N$.
\begin{lstlisting}
sage: Q=sol[1]
....: S=Q^(-1)
....: v1=matrix(CC,[[1],[1],[0]])
....: v2=matrix(CC,[[0],[1],[1]])
....: v3=matrix(CC,[[-1],[-1],[-1]])
....: v4=matrix(CC,[[1],[0],[0]])
....: v5=matrix(CC,[[0],[1],[0]])
....: v6=matrix(CC,[[0],[0],[1]])
....: w1=matrix(CC,[[1,0,0]])
....: w2=matrix(CC,[[0,1,0]])
....: w3=matrix(CC,[[0,0,1]])
....: w4=matrix(CC,[[1,-1,0]])
....: w5=matrix(CC,[[0,1,-1]])
....: w6=matrix(CC,[[-1,0,1]])
....: L1=[v1,v2,v3,v4,v5,v6]
....: L2=[w1,w2,w3,w4,w5,w6]
....: L3=[]
....: for i in range (0,6):
....:     v=S*L1[i]
....:     for j in range (0,6):
....:         w=L2[j]*Q
....:         c3=w[0][i0]*v[i0][0]
....:         c2=w[0][i1]*v[i1][0]
....:         c=log(float(abs(c3/c2)))
....:         L3.append(c)
....: max(L3)
-0.8301418502969936
\end{lstlisting}
   Thus, $\mathrm{log}\abs[\Big]{\frac{c_3}{c_2}}<-0.8$.
   Therefore, the above inequality implies
   \begin{align*}
    -3.3\cdot10^{16}\cdot\frac{\pi}{6}\cdot7\cdot\frac{\pi}{6}\cdot\mathrm{max}\{\mathrm{log}\abs{2k},\mathrm{log}(N),1\}<0.7-0.8-0.42N
   \end{align*}
   and so
   \begin{align*}
    -6.4\cdot10^{16}\cdot\mathrm{log}(2N)<-0.1-0.42N
   \end{align*}
   and this contradicts when $N>7\cdot10^{18}$.
  \end{proof}
  Thus, the condition (\ref{(iii)}) should be checked for only $n\leq7\cdot10^{18}$.
  The first three terms of each sequence are stated below.
\begin{lstlisting}
sage: A=matrix([[56,-19,-17],[-16,6,5],[207,-71,-63]])
....: v1=matrix([[1],[1],[0]])
....: v2=matrix([[0],[1],[1]])
....: v3=matrix([[-1],[-1],[-1]])
....: v4=matrix([[1],[0],[0]])
....: v5=matrix([[0],[1],[0]])
....: v6=matrix([[0],[0],[1]])
....: w1=matrix([[1,0,0]])
....: w2=matrix([[0,1,0]])
....: w3=matrix([[0,0,1]])
....: w4=matrix([[1,-1,0]])
....: w5=matrix([[0,1,-1]])
....: w6=matrix([[-1,0,1]])
....: L1=[v1,v2,v3,v4,v5,v6]
....: L2=[w1,w2,w3,w4,w5,w6]
....: L3=[]
....: for i in range (0,6):
....:     v=L1[i]
....:     for j in range (0,6):
....:         w=L2[j]
....:         a0=w*v
....:         a1=w*A*v
....:         a2=w*A*A*v
....:         L=[a0,a1,a2]
....:         L3.append(L)
....: print(L3)
[[[1], [37], [-50]], [[1], [-10], [28]], [[0], [136], [-199]], [[0], [47], [-78]], [[1], [-146], [227]], [[-1], [99], [-149]], [[0], [-36], [53]], [[1], [11], [-28]], [[1], [-134], [209]], [[-1], [-47], [81]], [[0], [145], [-237]], [[1], [-98], [156]], [[-1], [-20], [26]], [[-1], [5], [-15]], [[-1], [-73], [104]], [[0], [-25], [41]], [[0], [78], [-119]], [[0], [-53], [78]], [[1], [56], [-79]], [[0], [-16], [43]], [[0], [207], [-313]], [[1], [72], [-122]], [[0], [-223], [356]], [[-1], [151], [-234]], [[0], [-19], [29]], [[1], [6], [-15]], [[0], [-71], [114]], [[-1], [-25], [44]], [[1], [77], [-129]], [[0], [-52], [85]], [[0], [-17], [24]], [[0], [5], [-13]], [[1], [-63], [95]], [[0], [-22], [37]], [[-1], [68], [-108]], [[1], [-46], [71]]]
\end{lstlisting}
  Now the sequences are all satisfying the recurrence relation $a_n=-a_{n-1}+a_{n-3}$.\par
  We number these sequences from $1$ to $36$ in order.
  The sequences in the $1, 2, 5, 6, 8, 9, 10$, $12, 13, 14, 15, 19, 22, 24, 26, 28, 29, 33, 35, 36$-th places have non-zero initial terms, and the others start from zero.
  When considering these sequences in $\mathbb{Z}/p\mathbb{Z}$, then, by the pigeonhole principle, the sequences have some repetitive part and by the backward relation $a_{n-3}=a_{n}+a_{n-1}$, the repetitive part starts with the initial term.\par  
  First, we consider the former cases.
\begin{lstlisting}
sage: S=[1,2,5,6,8,9,10,12,13,14,15,19,22,24,26,28,29,33,35,36]
....: s=len(S)
....: L4=[]
....: for p in range (5,60):
....:     L4=[p]
....:     for i in range (0,s):
....:         v=L3[S[i]-1]
....:         a=v[0]%p
....:         b=v[1]%p
....:         c=v[2]%p
....:         a0=a
....:         b0=b
....:         c0=c
....:         A=0
....:         B=0
....:         C=0
....:         j=3
....:         L=[]
....:         while A!=a0 or B!=b0 or C!=c0:
....:             A=b
....:             B=c
....:             C=(a-c)%p
....:             a=A
....:             b=B
....:             c=C
....:             if c==0:
....:                 L.append(j)
....:             if L!=[]:
....:                 break
....:             j=j+1
....:         if L==[]:
....:             L4.append(S[i])
....:     if L4!=[p]:
....:         print(L4)
[20, 1, 2, 33]
[28, 1, 5, 9, 12, 13, 14, 26, 36]
[35, 2, 5, 6, 8, 9, 12, 14, 26, 29, 33]
[40, 1, 2, 12, 13, 14, 33]
[43, 10, 14, 19, 26]
[44, 9, 10, 12]
[45, 6, 15, 33]
[55, 9, 10, 12, 15, 22, 28, 33]
[56, 1, 2, 5, 8, 9, 12, 13, 14, 15, 26, 36]
[59, 2, 6, 9, 10, 19, 22, 24, 29, 35]
\end{lstlisting}
  This means that the zero term does not appear in these sequences when considering in $\mathbb{Z}/p\mathbb{Z}$ for some $p\in\mathbb{Z}$, and so in $\mathbb{Z}$.\par
  Next, we consider the latter cases.
\begin{lstlisting}
sage: S=[3,4,7,11,16,17,18,20,21,23,25,27,30,31,32,34]
....: s=len(S)
....: for i in range (0,s):
....:     S1=[1]
....:     for p in range(5,2000):
....:         k=S[i]
....:         v=L3[k-1]
....:         a=v[0]%p
....:         b=v[1]%p
....:         c=v[2]%p
....:         a0=a
....:         b0=b
....:         c0=c
....:         A=0
....:         B=0
....:         C=0
....:         j=3
....:         L=[]
....:         while A!=a0 or B!=b0 or C!=c0:
....:             A=b
....:             B=c
....:             C=(a-c)%p
....:             a=A
....:             b=B
....:             c=C
....:             if c==0:
....:                 L.append(j)
....:             if a==0:
....:                 if B!=b0 or C!=c0:
....:                     L.append(0)
....:                     L.append(0)
....:                     break
....:             j=j+1
....:         if len(L)==1:
....:             S1[0]=lcm(S1[0],j-3)
....:         if S1[0]>7*10**18:
....:              print(S[i],S1[0])
....:              break
3 197856007040168436960
4 3182657909595174410400
7 402266188667773029600
11 1028275312686859036800
16 954233501342344423200
17 1867369751282514679200
18 65028548896575818400
20 46999062546010454880
21 88051406847705100800
23 208338771542040266400
25 121699448915896051200
27 2006897833407314564640
30 12844334653156092240
31 244395985805903131200
32 12857036988550154400
34 27221868362904415200
\end{lstlisting}
  For a fixed sequence, the zero term of the sequence appears only once in the repetitive part when considering in $\mathbb{Z}/p\mathbb{Z}$ for some $p$.
  The above result means that the least common multiple of the lengths of such repetitive parts can be greater than $7\cdot10^{18}$.\par
  To conclude, by combining with Lemma \ref{zero terms of linear recurrence}, the condition (\ref{(iii)}) is confirmed.
  Thus, the assumption in Proposition \ref{equation for dynamical degree} holds for $A_1=\begin{pmatrix}
                                                                                          56 & -19 & -17 \\ -16 & 6 & 5 \\ 207 & -71 & -63
                                                                                         \end{pmatrix}$, and $\lambda_1(f_{A_1})$ is transcendental.
  Now
  \begin{align*}
   h_{A_1}:[x_0:x_1:x_2:x_3]\mapsto[x_0^{73}x_1^{16}x_2^{71}x_3^{63}:x_0^{53}x_1^{72}x_2^{52}x_3^{46}:x_0^{78}x_2^{77}x_3^{68}:x_1^{223}].
  \end{align*}
  By
  \footnotesize
  \begin{align*}
   \lambda_1(f_{A_1})\geq\Psi_{\mathcal{U},\mathcal{V}}(A_1)&=\sum_{v\in\mathcal{V}}\max_{u\in\mathcal{U}}\langle u,A_1v\rangle\\
                                 &\hspace{-2cm}=\max_{u\in\mathcal{U}}\left\langle u,\begin{pmatrix}
                                                                     37 \\ -10 \\ 136 
                                                                    \end{pmatrix}\right\rangle
                                   +\max_{u\in\mathcal{U}}\left\langle u,\begin{pmatrix}
                                                                     -36 \\ 11 \\ -134
                                                                    \end{pmatrix}\right\rangle
                                   +\max_{u\in\mathcal{U}}\left\langle u,\begin{pmatrix}
                                                                     -37 \\ 10 \\ -136
                                                                    \end{pmatrix}\right\rangle
                                   +\max_{u\in\mathcal{U}}\left\langle u,\begin{pmatrix}
                                                                     36 \\ -11 \\ 134
                                                                    \end{pmatrix}\right\rangle\\
                                  &\hspace{-2cm}=10+134+136+11=291
  \end{align*}
  \normalsize
  and
  \begin{align*}
   \lambda_1(f_{A_1})\leq\mathrm{deg}(f_{A_1})&\leq\mathrm{deg}(L_{B^{-1}})\cdot\mathrm{deg}(h_{-I})\cdot\mathrm{deg}(L_{B})\cdot\mathrm{deg}(h_{A_1})\\
                                      &=1\cdot3\cdot1\cdot223=669,
  \end{align*}
  it is calculated that $291\leq\lambda_1(f_{A_1})\leq669$.
 \subsection{The calculation of $\lambda_2(f_{A_1})$}\label{Calculation of lambda_2}
  As in Section \ref{Dynamical degrees} and Section \ref{2-cohomologically hyperbolic birational map}, $\lambda_2(f_{A_1})=\lambda_1(f_{A_1^{-1}})$.
  Now
  \begin{align*}
   A_1^{-1}=\begin{pmatrix}
           -23 & 10 & 7 \\ 27 & -9 & -8 \\ -106 & 43 & 32
          \end{pmatrix}.
  \end{align*}
  As Lemma \ref{condition for the recurrence}, we confirm the next condition.
  \begin{lemma}\label{condition for the recurrence2}
   Any element of $A_1^{-n}(\mathcal{V}\cup\mathcal{P})$ is contained in the interior of a $3$-dimensional cone generated by elements of $\mathcal{P}$ for all $n\in\mathbb{Z}_{>0}$,
  \end{lemma}
  \begin{proof}
   As in the proof of Lemma \ref{condition for the recurrence}, half of the cases for the elements of $\mathcal{V}$ are needless to consider.
   Now $A_1^{-1}$ has three eigenvalues
   \begin{align*}
    \tau_1=1.3247\cdots,\ \tau_2=-0.6623\cdots+0.5622\cdots i,\ \tau_3=-0.6623\cdots-0.5622\cdots i.
   \end{align*}
   For these eigenvalues, the elements of $\mathcal{V}$ (the half is omitted) and $\mathcal{P}$ are decomposed into the eigenspaces as below:
   \footnotesize
   \begin{align*}
    &v_1=\begin{pmatrix}
     1 \\ 1 \\ 0 
    \end{pmatrix}=\begin{pmatrix}
                   5.0493\cdots \\ 3.6579\cdots \\ 12.3205\cdots 
                  \end{pmatrix}+
         	  \begin{pmatrix}
                   -2.0246\cdots+19.8931\cdots i \\ -1.3289\cdots-10.1317\cdots i \\ -6.1602\cdots+77.7922\cdots i
                  \end{pmatrix}+
	 	  \begin{pmatrix}
                   -2.0246\cdots-19.8931\cdots i \\ -1.3289\cdots+10.1317\cdots i \\ -6.1602\cdots-77.7922\cdots i
                  \end{pmatrix}\\
    &v_2=\begin{pmatrix}
     0 \\ 1 \\ 1 
    \end{pmatrix}=\begin{pmatrix}
                   -3.1608\cdots \\ -2.2898\cdots \\ -7.7125\cdots 
                  \end{pmatrix}+
		  \begin{pmatrix}
                   1.5804\cdots-20.7021\cdots i \\ 1.6449\cdots+10.4819\cdots i \\ 4.3562\cdots-80.9097\cdots i
                  \end{pmatrix}+
   		  \begin{pmatrix}
                   1.5804\cdots+20.7021\cdots i \\ 1.6449\cdots-10.4819\cdots i \\ 4.3562\cdots+80.9097\cdots i
                  \end{pmatrix}\\
    &v'_1=\begin{pmatrix}
     -1 \\ -1 \\ -1 
    \end{pmatrix}=\begin{pmatrix}
                   -3.2374\cdots \\ -2.3453\cdots \\ -7.8995\cdots 
                  \end{pmatrix}+
		  \begin{pmatrix}
                   1.1187\cdots-10.4669\cdots i \\ 0.6726\cdots+5.3371\cdots i \\ 3.4497\cdots-40.9357\cdots i
                  \end{pmatrix}+
   		  \begin{pmatrix}
                   1.1187\cdots+10.4669\cdots i \\ 0.6726\cdots-5.3371\cdots i \\ 3.4497\cdots+40.9357\cdots i
                  \end{pmatrix}\\
    &v'_2=\begin{pmatrix}
     1 \\ 0 \\ 0 
    \end{pmatrix}=\begin{pmatrix}
                   6.3982\cdots \\ 4.6351\cdots \\ 15.6121\cdots 
                  \end{pmatrix}+
		  \begin{pmatrix}
                   -2.6991\cdots+31.1691\cdots i \\ -2.3175\cdots-15.8191\cdots i \\ -7.8060\cdots+121.8454\cdots i
                  \end{pmatrix}+
   		  \begin{pmatrix}
                   -2.6991\cdots-31.1691\cdots i \\ -2.3175\cdots+15.8191\cdots i \\ -7.8060\cdots-121.8454\cdots i
                  \end{pmatrix}\\
    &v'_3=\begin{pmatrix}
     0 \\ 1 \\ 0 
    \end{pmatrix}=\begin{pmatrix}
                   -1.3489\cdots \\ -0.9772\cdots \\ -3.2915\cdots 
                  \end{pmatrix}+
		  \begin{pmatrix}
                   0.6744\cdots-11.2759\cdots i \\ 0.9886\cdots+5.6873\cdots i \\ 1.6457\cdots-44.0532\cdots i
                  \end{pmatrix}+
   		  \begin{pmatrix}
                   0.6744\cdots+11.2759\cdots i \\ 0.9886\cdots-5.6873\cdots i \\ 1.6457\cdots+44.0532\cdots i
                  \end{pmatrix}\\
    &v'_4=\begin{pmatrix}
     0 \\ 0 \\ 1 
    \end{pmatrix}=\begin{pmatrix}
                   -1.8118\cdots \\ -1.3125\cdots \\ -4.4210\cdots 
                  \end{pmatrix}+
         	  \begin{pmatrix}
                   0.9059\cdots-9.4262\cdots i \\ 0.6562\cdots+4.7945\cdots i \\ 2.7105\cdots-36.8565\cdots i
                  \end{pmatrix}+
   	          \begin{pmatrix}
                   0.9059\cdots+9.4262\cdots i \\ 0.6562\cdots-4.7945\cdots i \\ 2.7105\cdots+36.8565\cdots i
                  \end{pmatrix}
   \end{align*}
   \normalsize
   We should only prove that the recurrence of the points by $A_1^{-1}$ satisfies
   \begin{align*}
    (\ast)\  x\neq0 , y\neq0 , z\neq0 , x\neq y , y\neq z \text{ and } z\neq x.
   \end{align*}
   Here, we prove only for $v_1$.
   The equation
   \footnotesize
   \begin{align*}
    A_1^{-n}(v_1)=\begin{pmatrix}
                 X_n \\ Y_n \\ Z_n 
                \end{pmatrix}=\tau_1^n\begin{pmatrix}
                                    5.0493\cdots \\ 3.6579\cdots \\ 12.3205\cdots 
                                   \end{pmatrix}
         	             +\tau_2^n\begin{pmatrix}
                                    -2.0246\cdots+19.8931\cdots i \\ -1.3289\cdots-10.1317\cdots i \\ -6.1602\cdots+77.7922\cdots i
                                   \end{pmatrix}
			     +\tau_3^n\begin{pmatrix}
                                    -2.0246\cdots-19.8931\cdots i \\ -1.3289\cdots+10.1317\cdots i \\ -6.1602\cdots-77.7922\cdots i
                                   \end{pmatrix}                
   \end{align*}
   \normalsize
   means that the coordinates in the first-term growth to $\pm\infty$, and the coordinates in the other terms converge to $0$.
   Write
   \footnotesize
   \begin{gather*}
    \tau_1^n\begin{pmatrix}
        5.0493\cdots \\ 3.6579\cdots \\ 12.3205\cdots 
       \end{pmatrix}=\begin{pmatrix}
                      x_n \\ y_n \\ z_n 
                     \end{pmatrix},\\
    \tau_2^n\begin{pmatrix}
          -2.0246\cdots+19.8931\cdots i \\ -1.3289\cdots-10.1317\cdots i \\ -6.1602\cdots+77.7922\cdots i
         \end{pmatrix}=\begin{pmatrix}
                        x'_n \\ y'_n \\ z'_n 
                       \end{pmatrix},
    \tau_3^n\begin{pmatrix}
          -2.0246\cdots-19.8931\cdots i \\ -1.3289\cdots+10.1317\cdots i \\ -6.1602\cdots-77.7922\cdots i
         \end{pmatrix} =\begin{pmatrix}
                         x''_n \\ y''_n \\ z''_n 
                        \end{pmatrix}
   \end{gather*}
   \normalsize
   with $X_n=x_n+x'_n+x''_n$, $Y_n=y_n+y'_n+y''_n$, $Z_n=z_n+z'_n+z''_n$ and then
   \begin{align*}
    \min\{\abs{x_n},\abs{y_n},\abs{z_n},\abs{x_n-y_n},\abs{y_n-z_n},\abs{z_n-x_n}\}=\abs{\tau_1}^n\cdot1.3913\cdots,\\
    \max\{\abs{x'_n},\abs{y'_n},\abs{z'_n},\abs{x'_n-y'_n},\abs{y'_n-z'_n},\abs{z'_n-x'_n}\}=\abs{\tau_2}^n\cdot88.0565\cdots\\
    \max\{\abs{x''_n},\abs{y''_n},\abs{z''_n},\abs{x''_n-y''_n},\abs{y''_n-z''_n},\abs{z''_n-x''_n}\}=\abs{\tau_3}^n\cdot88.0565\cdots
   \end{align*}
   and so for $n\geq12$,
   \begin{align*} 
    \min\{\abs{X_n},\abs{Y_n},\abs{Z_n},\abs{X_n-Y_n},\abs{Y_n-Z_n},\abs{Z_n-X_n}\}\geq8.0500\cdots>0.
   \end{align*}
   The case for $n\leq11$ can be checked by calculation, and so the condition ($\ast$) is true for $A_1^{-n}(v_1)$ and the others are confirmed similarly.
  \end{proof}
  The first few terms of the sequences $A_1^{-k}v$ for $v\in\mathcal{V}$ are as below ($k=1,2,3,\ldots$):
  \footnotesize
  \begin{align*}
   &A_1^{-k}v_1=\begin{pmatrix}
               -13 \\ 18 \\ -63 
              \end{pmatrix},\begin{pmatrix}
               38 \\ -9 \\ 136 
              \end{pmatrix},\begin{pmatrix}
               -12 \\ 19 \\ -63
              \end{pmatrix},\begin{pmatrix}
               25 \\ 9 \\ 73
              \end{pmatrix},\begin{pmatrix}
               26 \\ 10 \\ 73
              \end{pmatrix},\begin{pmatrix}
               13 \\ 28 \\ 10
              \end{pmatrix},\begin{pmatrix}
               51 \\ 19 \\ 146
              \end{pmatrix},\begin{pmatrix}
               39 \\ 38 \\ 83
              \end{pmatrix},\begin{pmatrix}
               64 \\ 47 \\ 156
              \end{pmatrix},\begin{pmatrix}
               90 \\ 57 \\ 229
              \end{pmatrix},
	      \ldots\\
   &A_1^{-k}v_2=\begin{pmatrix}
               17 \\ -17 \\ 75 
              \end{pmatrix},\begin{pmatrix}
               -36 \\ 12 \\ -133
              \end{pmatrix},\begin{pmatrix}
               17 \\ -16 \\ 76
              \end{pmatrix},\begin{pmatrix}
               -19 \\ -5 \\ -58
              \end{pmatrix},\begin{pmatrix}
               -19 \\ -4 \\ -57
              \end{pmatrix},\begin{pmatrix}
               -2 \\ -21 \\ 18
              \end{pmatrix},\begin{pmatrix}
               -38 \\ -9 \\ -115
              \end{pmatrix},\begin{pmatrix}
               -21 \\ -25 \\ -39
              \end{pmatrix},\begin{pmatrix}
               -40 \\ -30 \\ -97
              \end{pmatrix},
	      \ldots\\
   &A_1^{-k}(-v_1)=\begin{pmatrix}
               13 \\ -18 \\ 63 
              \end{pmatrix},\begin{pmatrix}
               -38 \\ 9 \\ -136 
              \end{pmatrix},\begin{pmatrix}
               12 \\ -19 \\ 63
              \end{pmatrix},\begin{pmatrix}
               -25 \\ -9 \\ -73
              \end{pmatrix},\begin{pmatrix}
               -26 \\ -10 \\ -73
              \end{pmatrix},\begin{pmatrix}
               -13 \\ -28 \\ -10
              \end{pmatrix},\begin{pmatrix}
               -51 \\ -19 \\ -146
              \end{pmatrix},\begin{pmatrix}
               -39 \\ -38 \\ -83
              \end{pmatrix},\begin{pmatrix}
               -64 \\ -47 \\ -156
              \end{pmatrix},
	      \ldots\\
   &A_1^{-k}(-v_2)=\begin{pmatrix}
               -17 \\ 17 \\ -75 
              \end{pmatrix},\begin{pmatrix}
               36 \\ -12 \\ 133
              \end{pmatrix},\begin{pmatrix}
               -17 \\ 16 \\ -76
              \end{pmatrix},\begin{pmatrix}
               19 \\ 5 \\ 58
              \end{pmatrix},\begin{pmatrix}
               19 \\ 4 \\ 57
              \end{pmatrix},\begin{pmatrix}
               2 \\ 21 \\ -18
              \end{pmatrix},\begin{pmatrix}
               38 \\ 9 \\ 115
              \end{pmatrix},\begin{pmatrix}
               21 \\ 25 \\ 39
              \end{pmatrix},\begin{pmatrix}
               40 \\ 30 \\ 97
              \end{pmatrix},\begin{pmatrix}
               59 \\ 34 \\ 154
              \end{pmatrix},
	      \ldots
  \end{align*}
  \normalsize
  Similar to the proof of Lemma \ref{condition for the recurrence2}, we can confirm that $\langle u,A_1^{-k}v_1\rangle$ are all maximized by $u={\footnotesize\begin{pmatrix}0 \\ 0 \\ 0\end{pmatrix}}$ for all $k\geq4$, $\langle u,A_1^{-k}v_2\rangle$ are all maximized by $u={\footnotesize\begin{pmatrix}0 \\ 0 \\ -1\end{pmatrix}}$ for all $k\geq7$, $\langle u,A_1^{-k}(-v_1)\rangle$ are all maximized by $u={\footnotesize\begin{pmatrix}0 \\ 0 \\ -1\end{pmatrix}}$ for all $k\geq7$ and $\langle u,A_1^{-k}(-v_2)\rangle$ are all maximized by $u={\footnotesize\begin{pmatrix}0 \\ 0 \\ 0\end{pmatrix}}$ for all $k\geq7$.
  Thus,
  \begin{align*}
   &\max_{u\in\mathcal{U}}\langle u,A_1^{-k}v_1\rangle=63,9,63,0,0,0,0,0,0,0,\ldots(k=1,2,3\ldots)\\
   &\max_{u\in\mathcal{U}}\langle u,A_1^{-k}v_2\rangle=17,133,16,58,57,21,115,39,97,154,\ldots(k=1,2,3,\ldots)\\
   &\max_{u\in\mathcal{U}}\langle u,A_1^{-k}(-v_1)\rangle=18,136,19,73,73,28,146,83,156,229,\ldots(k=1,2,3,\ldots)\\
   &\max_{u\in\mathcal{U}}\langle u,A_1^{-k}(-v_2)\rangle=75,12,76,0,0,18,0,0,0,0,\ldots(k=1,2,3,\ldots)
  \end{align*}
  and for $k\geq10$, these sequences have recurrence relations with $a_k=a_{k-2}+a_{k-3}$, associated with $x^3-x-1$, the characteristic polynomial of $A_1^{-1}$.
  Thus, for $k\geq10$, $P_k=\Psi_{\mathcal{U},\mathcal{V}}(A_1^{-k})$ is also a linear recurrence that satisfies $P_k=P_{k-2}+P_{k-3}$, and the first few terms are
  \begin{align*}
   173,290,174,131,130,67,261,122,253,383,\ldots.
  \end{align*}
  Now $\lambda''=\lambda_1(f_{A_1^{-1}})$ satisfies the equation
  \begin{align*}
   1=\sum_{n=1}^{\infty} \frac{P_n}{{\lambda''}^n}
  \end{align*}
  by Proposition \ref{equation for dynamical degree} and this implies
  \begin{align*}
   &\lambda''-173=\sum_{n=1}^{\infty} \frac{P_{n+1}}{{\lambda''}^n}\\
   &{\lambda''}^2-173\lambda''-290=\sum_{n=1}^{\infty} \frac{P_{n+2}}{{\lambda''}^n}\\
   &{\lambda''}^3-173{\lambda''}^2-290\lambda''-174=\sum_{n=1}^{\infty} \frac{P_{n+3}}{{\lambda''}^n}
  \end{align*}
  By using the equation $P_{n+3}=P_{n+1}+P_n$ for $n\geq7$,
  \begin{align*}
   {\lambda''}^3-173{\lambda''}^2-290\lambda''-174=(\lambda''-173)+1-\left(\frac{332}{\lambda''}+\frac{334}{{\lambda''}^2}+\frac{238}{{\lambda''}^3}+\frac{75}{{\lambda''}^5}+\frac{75}{{\lambda''}^6}\right)
  \end{align*}
  and so
  \begin{align*}
   {\lambda''}^9-173{\lambda''}^8-291{\lambda''}^7-2{\lambda''}^6+332{\lambda''}^5+334{\lambda''}^4+238{\lambda''}^3+75\lambda''+75=0
  \end{align*}
  Thus,
  \begin{align*}
   \lambda''=174.6660\cdots.
  \end{align*}
  To conclude,
  \begin{align*}
   \lambda_0(f_{A_1})=1,\ 291\leq\lambda_1(f_{A_1})\leq669,\ \lambda_2(f_{A_1})=174.6660\cdots,\ \lambda_3(f_{A_1})=1
  \end{align*}
  and $\lambda_1(f_{A_1})$ is transcendental.
  Thus, $f_{A_1}$ is $1$-cohomologically hyperbolic, and so by Theorem \ref{MW22}, there exists a point $x\in\mathbb{P}^3(\overline{\mathbb{Q}})_{f_{A_1}}$, with Zariski dense forward orbit, such that $\alpha_{f_{A_1}}(x)$ exists, and this equals to $\lambda_1(f_{A_1})$.
  Thus, $\alpha_f(x)$ is transcendental, and Theorem \ref{Main theorem} and Corollary \ref{corollary of Main theorem} for the case $d=3$ are proved.
  The rest of Corollary \ref{corollary of Main theorem} is stated below again.
  \begin{corollary}\label{transcendental arithmetic degree}
   For any $d\geq3$, there exist a birational map $\phi:\mathbb{P}^d\dashrightarrow\mathbb{P}^d$ over $\overline{\mathbb{Q}}$ and a point $P\in\mathbb{P}^{d}(\overline{\mathbb{Q}})_\phi$ such that its arithmetic degree $\alpha_\phi(P)$ is transcendental.
  \end{corollary}
  \begin{proof}
   By Theorem \ref{Main theorem}, there exist a birational map $f:\mathbb{P}^3\dashrightarrow\mathbb{P}^3$ and a point $x\in\mathbb{P}^3(\overline{\mathbb{Q}})_{f}$ with transcendental arithmetic degree $\alpha_f(x)$.
   $f$ induces a birational map
   \begin{align*}
   \begin{array}{cccc}
    \tilde{\phi}=f\times \id:&\mathbb{P}^3\times\mathbb{P}^{d-3}&\dashrightarrow&\mathbb{P}^3\times\mathbb{P}^{d-3}\\
    & ([x_0:x_1:x_2:x_3],\ [1:\cdots:1]) & \mapsto & (f([x_0:x_1:x_2:x_3]),\ [1:\cdots:1])
   \end{array}
   \end{align*}
   and the forward orbit of $(x,\ [1:\cdots:1])$ by $\tilde{\phi}$ is well-defined.
   The birational map
   \begin{align*}
   \begin{array}{cccc}
    &\mathbb{P}^3\times\mathbb{P}^{d-3}&\dashrightarrow&\mathbb{P}^d\\
    & ([x_0:x_1:x_2:x_3],\ [y_0:\cdots:y_{d-3}]) & \mapsto & [x_0y_0:x_1y_0:x_2y_0:x_3y_0:x_3y_1:\cdots:x_3y_{d-3}]
   \end{array}
   \end{align*}
   induces the birational map $\phi:\mathbb{P}^d\dashrightarrow\mathbb{P}^d$.\par
   Now the forward orbit of the point $P\in\mathbb{P}^d$, which corresponds to $(x,\ [1:\cdots:1])\in\mathbb{P}^3\times\mathbb{P}^{d-3}$, is well-defined.\par
   By the definition of a Weil height and the arithmetic degree, $\alpha_\phi(P)=\alpha_f(x)$ holds, and it is transcendental.
  \end{proof}
  \begin{remark}\label{maximal dynamical degree is transcendental}
   The last part of the proof is also proved by using \cite[Lemma 3.3]{Sil17} and \cite[Theorem 3.4]{MSS18}.
   Also, in the above proof, Theorem \ref{birational conjugate} and \cite[Lemma 3.3]{Sil17} imply the first dynamical degree of $\phi$ is calculated as
   \begin{align*}
    \lambda_1(\phi)=\lambda_1(\tilde{\phi})=\lambda_1(f\times id)=\max\{\lambda_1(f),1\}=\lambda_1(f).
   \end{align*}\par
   Thus, by constructing the map $\phi$ from $f=f_{A_1}$ in Section \ref{1-cohomologically hyperbolic birational map}, we can make $\phi$ as $\lambda_1(\phi)=\lambda_1(f_{A_{1}})=\alpha_\phi(P)$ in Corollary \ref{transcendental arithmetic degree}.
   Denote $\mu=\lambda_1(f_{A_1})$, $\nu=\lambda_2(f_{A_1})$, and by using \cite[Theorem 1.1]{DN11} as in \cite[Example 3.7]{DN11}, \cite[Corollary 2.5]{San20}, the dynamical degrees of $\phi$ are calculated as
   \begin{align*}
    \lambda_0(\phi)=1,\ \lambda_1(\phi)=\mu,\ldots,\lambda_{d-2}(\phi)=\mu,\ \lambda_{d-1}(\phi)=\nu,\ \lambda_{d}(\phi)=1
   \end{align*}
   by the inequality $1=\lambda_0(f_{A_1})<\lambda_1(f_{A_1})>\lambda_2(f_{A_1})>\lambda_3(f_{A_1})=1$ (and now, $\mu$ is transcendental, and $\phi$ is not cohomologically hyperbolic for $d>3$).
  \end{remark}
  \begin{remark}\label{simultaneously}
   Both the examples in Section \ref{2-cohomologically hyperbolic birational map} and Section \ref{1-cohomologically hyperbolic birational map} have an algebraic second dynamical degree.
   Define the birational map $f_A:=L_{B^{-1}}\circ h_{-I}\circ L_{B}\circ h_A$ for $A\in\mathrm{GL}_{3}(\mathbb{Z})$ and $B$ as in Proposition \ref{equation for dynamical degree}.
   Assume $A$ satisfies the condition in Proposition \ref{equation for dynamical degree}, with the irreducible characteristic polynomial $p(x)\in\mathbb{Z}[x]$.
   Let $z_1,z_2$ and $z_3$ be the roots of $p(x)$, which are ordered as $\abs{z_1}\geq\abs{z_2}\geq\abs{z_3}$ with $\abs{z_1z_2z_3}=1$.\par
   As well as the computation in Section \ref{2-cohomologically hyperbolic birational map} and Section \ref{1-cohomologically hyperbolic birational map}, $\lambda_1(f_A)$ is algebraic if $\abs{z_1}>\abs{z_2}\geq\abs{z_3}$.\par
   The method in this paper only consider the case that both $A$ and $A^{-1}$ satisfy the condition in Proposition \ref{equation for dynamical degree}, and this means that one of $\lambda_1(f_A)$ and $\lambda_2(f_A)$ is algebraic.
  \end{remark}

\renewcommand{\refname}{References}

\end{document}